\documentclass[a4paper, 11pt]{amsart}

\newcommand{\R}{\mathbf{R}}
\newcommand{\e}{\varepsilon}
\newcommand{\ra}{\rightarrow}
\newcommand{\ol}{\overline}
\renewcommand{\d}{\partial}

\renewcommand{\phi}{\varphi}   

\theoremstyle{cupthm}
\newtheorem{theorem}{Theorem}[section]
\newtheorem{proposition}[theorem]{Proposition}
\newtheorem{corollary}[theorem]{Corollary}
\newtheorem{lemma}[theorem]{Lemma}

\begin{document}
\title{Distinct solutions to generated Jacobian equations cannot intersect}
\author[C. Rankin]{Cale Rankin}
\email{cale.rankin@anu.edu.au}
\address{Australian National University}
\thanks{Supported by an Australian Government Research Training Program (RTP) Scholarship.\\
  Partially supported by Australian Research Council Grant DP180100431}

 \begin{abstract}
We prove that if two $C^{1,1}(\Omega)$ solutions of the second boundary value problem for the generated Jacobian equation intersect in $\Omega$ then they are the same solution. In addition we extend this result to $C^{2}(\ol{\Omega})$ solutions intersecting on the boundary, via an additional convexity condition on the target domain.   
 \end{abstract}

\maketitle
\section{Introduction}\label{sec:intro}

The prescribed Jacobian equation coupled with the second boundary value problem arises in optimal transport and geometric optics. These equations, with their boundary condition, take the form
\begin{align}
\label{eq:pje}  \det DY(\cdot,u,Du) &= \frac{f(x)}{f^*(Y(\cdot,u,Du))} \ \ \text{in } \Omega,\\
\label{eq:2bvp}  Y(\cdot,u,Du)(\Omega) &= \Omega^*,
\end{align}
where $Y:\R^n\times\R\times\R^n\ra\R^n$,  and the functions $f,f^*$ are positive densities on the prescribed domains $\Omega,\Omega^* \subset \R^n$. Such equations have not been profitably studied without additional structure on $Y$. In this paper we require that $Y$ arise from a generating function and thus work in the framework of generated Jacobian equations (GJE), which were introduced by Trudinger \cite{truloc}.  Since $Y$ depends on $u$ in an unknown way we no longer have uniqueness of solutions (even up to a constant). In this paper we prove a version of a uniqueness result: that distinct solutions cannot intersect at any point in the domain. 

\begin{theorem}\label{thm:main}
  Suppose $g$ is a generating function on $\Gamma$ satisfying A1,A1$^{*}$,A2 and $f,f^* > 0$ are $C^1$ and satisfy the mass balance condition \eqref{eq:mb}. Suppose $u,v \in C^{1,1}(\Omega)$ are $g$-convex generalized solutions of (\ref{eq:pje}) subject to (\ref{eq:2bvp}). If there is $x_0 \in \Omega$ such that $u(x_0) = v(x_0)$, then $u \equiv v$ in $\Omega$. 
\end{theorem}

Our plan is as follows: In Section \ref{sec:gen} we introduce the theory of generating functions and the definitions required to understand the statement of Theorem \ref{thm:main}. In Section \ref{sec:grad} we prove, using a lemma of Alexandrov's, that wherever solutions intersect they have the same gradient. We show in Section \ref{sec:harnack} a weak Harnack inequality that we use in Section \ref{sec:tang} to prove solutions intersecting in the interior of $\Omega$ are the same. Finally in Section 6 we give conditions which yield the same result when $x_0 \in \d\Omega$.

\section{Generated Jacobian equations and $g$-convexity}
\label{sec:gen}

The following framework is standard for GJE and mirrors \cite{JT2}. Further details on GJE may also be found in \cite{GK,Jhaveri}.  Let $\Gamma \subset \R^n\times\R^n\times \R$ be a domain for which the projections
\begin{equation}
  \label{eq:i}
  I(x,y):= \{z \in \R; (x,y,z) \in \Gamma\},
\end{equation}
are (possibly empty) open intervals. We consider a function $g \in C^4(\Gamma)$ which we assume satisfies the following properties.\\
\textbf{A1: }For each $(x,u,p)$ in $\mathcal{U}$, which is defined as 
\[ \mathcal{U} = \{(x,g(x,y,z),g_x(x,y,z)); (x,y,z) \in \Gamma\},\]
there exists a unique $(x,y,z)\in\Gamma$ such that
\begin{align*}
  &g(x,y,z) = u, && g_x(x,y,z) = p.
\end{align*}
\textbf{A1$^*$: } For each fixed $(y,z)$ the mapping
\[ x \ra \frac{-g_y}{g_z}(x,y,z),\] is one to one. \\
\textbf{A2: } $g_z<0$ and 
\[ E:= g_{x,y}-(g_z)^{-1}g_{x,z}\otimes g_y,\]
satisfies $\det E \neq 0$. 

Assumption A1 allows us to define mappings $Y:\mathcal{U} \ra \R^n$ and $Z:\mathcal{U}\ra \R$ by the requirement that they uniquely solve
\begin{align}
\label{eq:yzdef1}  g(x,Y(x,u,p),Z(x,u,p)) &= u,\\
\label{eq:yzdef2}  g_x(x,Y(x,u,p),Z(x,u,p)) &= p.
\end{align}
Herein we assume that $Y$ is the mapping appearing in \eqref{eq:pje} and \eqref{eq:2bvp}. This assumption allows us to rewrite \eqref{eq:pje} as a Monge-Amp\`ere type equation as follows. Setting $u=u(x)$, $p = Du(x)$ and differentiating \eqref{eq:yzdef1} with respect to the $j^{th}$ coordinate yields
\[g_{x_j}+g_{y_k}D_jY^k+g_zD_jZ = D_ju,\]
and since $g_x = Du$ we have
\begin{equation}
  \label{eq:matederive1}
  D_jZ = -\frac{1}{g_z}g_{y_k}D_jY^k. 
\end{equation}
Similarly differentiating \eqref{eq:yzdef2} yields
\begin{equation}
g_{x_i, x_j} +g_{x_i,y_k}D_jY^k+g_{x_i, z}D_jZ = D_{ij}u.\label{eq:2}
\end{equation}
We substitute \eqref{eq:matederive1} into \eqref{eq:2} and obtain
\[ (g_{x_i, y_k}-\frac{1}{g_z}g_{x_i,z}g_{y_k})D_jY^k = D_{ij}u-g_{x_ix_j}.\]
Thus, with $E$ as defined in A2,
\[ DY(x,u,Du) = E^{-1}[D^2u - g_{xx}(x,Y(x,u,Du),Z(x,u,Du))], \]
and we rewrite \eqref{eq:pje} as 
\begin{equation}
  \label{eq:mate}
  \det[D^2u - A(\cdot,u,Du)] = B(\cdot,u,Du),
\end{equation}
where
\begin{align*}
  A(\cdot,u,Du) &= g_{xx}(\cdot,Y(\cdot,u,Du),Z(\cdot,u,Du)),\\
  B(\cdot,u,Du) &= \det E \frac{f(x)}{f^*(Y(\cdot,u,Du))}.
\end{align*}
The PDE \eqref{eq:mate} is degenerate elliptic when $D^2u \geq g_{xx}$.

The assumptions on $g$ allow for the introduction of a convexity theory where $g$ plays the role of a supporting hyperplane. A function $u: \Omega \ra \R$ is called $g$-convex if for every $x_0\in\Omega$ there exists $y_0,z_0$ such that
  \begin{align}
\label{eq:gconvdef1}    u(x_0) &= g(x_0,y_0,z_0),\\
 \label{eq:gconvdef2}   u(x) &\geq g(x,y_0,z_0),
  \end{align}
for all $x \in \Omega$. We call $g(\cdot,y_0,z_0)$ a $g$-support at $x_0$. 

Suppose $u$ is a differentiable $g$-convex function and $g(\cdot,y_0,z_0)$ is a $g$-support at $x_0$. Then $x \mapsto u(x)-g(x,y_0,z_0)$ has a minimum at $x_0$. Hence $Du(x_0) = g_x(x_0,y_0,z_0)$ which, with \eqref{eq:gconvdef1}, implies via \eqref{eq:yzdef1} and (\ref{eq:yzdef2}), that $y_0 = Y(x_0,u(x_0),Du(x_0))$. Furthermore if $u$ is $C^2$ then $D^2u - g_{xx}$ is nonnegative definite and the equation is degenerate elliptic. 

In this article we work with generalized solutions. A definition of generalized solution exists for functions which are merely $g$-convex \cite{truloc}. However our results rely on differentiability so we give the definition of a differentiable $g$-convex generalized solution. A differentiable $g$-convex function $u:\Omega \ra \R$ is called a generalized solution of \eqref{eq:pje} if for every $E\subset \Omega$
\begin{equation}
  \label{eq:8}
  \int_{Y(\cdot,u,Du)(E)}f^*(y) \ dy = \int_Ef(x)\ dx,
\end{equation}
where $f^*$ is extended to 0 outside $\Omega^*$. If in addition $Y(\cdot,u,Du)(\Omega) \subset \ol{\Omega^*}$ we say $u$ is a generalized solution of \eqref{eq:pje} subject to \eqref{eq:2bvp}, that is, a generalized solution of the second boundary value problem. Note that under the mass balance condition
\begin{equation}
  \label{eq:mb}
  \int_{\Omega}f = \int_{\Omega^*}f^*,
\end{equation}
which is necessary for classical solvability, generalized solutions of the second boundary value problem satisfy
\begin{equation}
Y(\cdot,u,Du)(\Omega) = \ol{\Omega^*}\setminus \mathcal{Z} \label{eq:9},
\end{equation}
for some set $\mathcal{Z}$ of Lebesgue measure 0.

Moreover any generalized solution which is $C^{1,1}(\Omega)$ and thus twice differentiable almost everywhere satisfies both \eqref{eq:pje} and \eqref{eq:mate} almost everywhere in $\Omega$.

\section{Solutions have the same gradients where they intersect} \label{sec:grad}
In this section we show generalized solutions of \eqref{eq:pje} subject to (\ref{eq:2bvp}) satisfy $Du \equiv Dv$ on $\{x \in \Omega; u(x) = v(x) \}$. 
Our main tool is a lemma concerning arbitrary convex functions due to Alexandrov \cite{al} and used by McCann \cite[Lemma 13]{mccann} in the Monge-Amp\`ere case. We adapt McCann's proof to the $g$-convex case. We  use the notation $Y_u(x) = Y(x,u(x),Du(x))$ and similarly for $Y_v,Z_u,Z_v$.

\begin{lemma}\label{lem:aleksandrovmccann}
Assume $u,v:\Omega \ra \R$ are $g$-convex and differentiable. Suppose for some $x_0 \in \Omega$ there holds $u(x_0) = v(x_0)$ and $Du(x_0) \neq Dv(x_0)$. 
  With $\Omega':=\{x \in \Omega; u(x)>v(x)\}$ set $\Xi := Y_v^{-1}(Y_u(\Omega'))$. Then $\Xi \subset \Omega'$ and $x_0$ is a positive distance from $\Xi$. 
\end{lemma}
\begin{proof}
  We begin by proving the subset assertion. Take $\xi \in \Xi$. The definition of $\Xi$ implies there is $x \in \Omega'$ with $Y_v(\xi) = Y_u(x)$. We claim $Z_u(x) < Z_v(\xi)$. Indeed, were this not the case $Z_u(x) \geq Z_v(\xi)$, which when combined with $Y_v(\xi) = Y_u(x)$ and $g_z < 0$ yields that for any $z$
  \[ g(z,Y_u(x),Z_u(x)) \leq g(z,Y_v(\xi),Z_v(\xi)).\]
  This would imply
  \begin{align*}
    u(x) &= g(x,Y_u(x),Z_u(x))\\
         &\leq g(x,Y_v(\xi),Z_v(\xi))\\
    &\leq v(x),
  \end{align*}
  where the final inequality is because $g(\cdot,Y_v(\xi),Z_v(\xi))$ is a $g$-support. Since $x \in \Omega'$ this contradiction establishes $Z_u(x) < Z_v(\xi)$. Using this and $g_z<0$ we have for any $z$
  \begin{align}
   \nonumber u(z) &\geq g(z,Y_u(x),Z_u(x))\\
   \label{eq:ineq}  &> g(z,Y_v(\xi),Z_v(\xi)).
  \end{align}
  For $z = \xi$ we obtain $u(\xi) > g(\xi,Y_v(\xi),Z_v(\xi)) = v(\xi)$, implying $\xi \in \Omega'$ and establishing the subset relation.

  We move on to the distance claim. We suppose to the contrary that there exists a sequence of  $\{\xi_n\}_{n=1}^\infty$ in $\Xi$ with $\xi_n \ra x_0$. The definition of $\Xi$ implies for each $\xi_n$ there exists an $x_n \in \Omega'$ with $Y_v(\xi_n) = Y_u(x_n).$

  Now $Du(x_0) \neq Dv(x_0)$ implies in any neighbourhood of $x_0$ there is a particular $z$ for which 
  \begin{equation}
    \label{eq:m1}
    u(z) < g(z,Y_v(x_0),Z_v(x_0)),
  \end{equation}
  for if not we have
  \begin{align*}
    u(x_0) &= v(x_0) = g(x_0,Y_v(x_0),Z_v(x_0))\\
    u(x) &\geq g(x,Y_v(x_0),Z_v(x_0)) \text{ in a neighbourhood of }x_0.
  \end{align*}
  This implies $u(\cdot)-g(\cdot,Y_v(x_0),Z_v(x_0))$ has a local minimum at $x_0$. Thus
  \[ Du(x_0) = g_x(x_0,Y_v(x_0),Z_v(x_0)) = Dv(x_0),\]
  and this contradiction establishes \eqref{eq:m1}. 

  Since  our derivation of \eqref{eq:ineq} used only that $x \in \Omega'$ and $\xi \in \Xi$ satisfied $Y_v(\xi) = Y_u(x)$, (\ref{eq:ineq}) also holds for  for $x_n$ and $\xi_n$. That is for any $z$ we have
  \begin{equation}
  u(z) > g(z,Y_v(\xi_n),Z_v(\xi_n)) .\label{eq:m2}
\end{equation}

Combining \eqref{eq:m1} and \eqref{eq:m2} we obtain
  \[ g(z,Y_v(x_0),Z_v(x_0)) > u(z) >g(z,Y_v(\xi_n),Z_v(\xi_n)), \]
  which, on sending $\xi_n\ra x_0$ yields a contradiction and completes the proof of Lemma \ref{lem:aleksandrovmccann}. 
\end{proof}

We use this lemma to show solutions solutions have the same gradient where they intersect.

\begin{corollary}
Assume the conditions of Theorem \ref{thm:main}. Then $Du \equiv Dv$ on the set $\{x \in \Omega; u(x) = v(x)\}$.  
\end{corollary}
\begin{proof}
  Suppose otherwise. Then there is $x_0 \in \Omega$ with $u(x_0) = v(x_0)$ and $Du(x_0) \neq Dv(x_0)$. This implies any neighbourhood of $x_0$ contains a $z$ with $u(z) > v(z)$, which is to say $x_0 \in \d \Omega' \cap \Omega$. By the previous lemma, for $\e$ sufficiently small $B_\e(x_0) \cap \Xi = \emptyset$ and thus $\Xi \subset \Omega'\setminus B_\e(x_0).$ On the other hand, since $x_0 \in \d \Omega',$ and $u$ is continuous, $|B_\e(x_0) \cap \Omega'| > 0$. Hence
  \[ |Y_v^{-1}(Y_u(\Omega'))| = |\Xi| \leq  |\Omega' \setminus B_\e(x_0)| < |\Omega'|,\]
  and since $f^*$ is bounded below, this implies 
  \begin{equation}
   \int_{Y_v^{-1}(Y_u(\Omega'))}f^*(Y_v)\det DY_v \ dx  < \int_{\Omega'} f^*(Y_v)\det DY_v \ dx .\label{eq:4}
 \end{equation}

  The change of variables formula holds for the mappings $Y_u$ and $Y_v$ even though they may not be diffeomorphisms. The reasoning here is the same reasoning which yields the change of variables formula for the gradient of $C^{1,1}$ convex functions and uses the assumption A1$^*$ (see \cite[Theorem A.31]{FigalliBook} and \cite[\S 4]{truloc}). In light of this (\ref{eq:4}) yields the following contradiction:
  
 \begin{align}
    \label{eq:gen1}    \int_{\Omega'}f(x) \ dx  &= \int_{Y_u(\Omega')}f^*(y) \ dy\\
   \label{eq:gen2}                     &= \int_{Y_v(Y_v^{-1}(Y_u(\Omega')))}f^*(y) \ dy\\
   \nonumber                                   &= \int_{Y_v^{-1}(Y_u(\Omega'))}f^*(Y_v)\det DY_v \ dy\\
   \nonumber                          &< \int_{\Omega'} f^*(Y_v)\det DY_v \ dy= \int_{\Omega'}f(x) \ dx.
  \end{align}
  Here the equality between \eqref{eq:gen1} and \eqref{eq:gen2} uses the generalized boundary condition in conjunction with \eqref{eq:9} to deduce
  \[ Y_v(Y_v^{-1}(Y_u(\Omega'))) = Y_v(\Omega) \cap Y_u(\Omega') = Y_u(\Omega')\setminus\mathcal{Z},\]
  for some set $\mathcal{Z}$ with Lebesgue measure 0, hence the integrals over these sets are equal. 
\end{proof}

 \section{A weak Harnack inequality}
\label{sec:harnack}

\begin{proposition}\label{prop:harn} Suppose $u,v \in C^{1,1}(\Omega)$ satisfy (\ref{eq:pje}) almost everywhere and $u\geq v $ in $\Omega$.  Then for any $\tilde{\Omega} \subset\subset \Omega$ there exists $p,C>0$ such that
\begin{equation}
\Big(\frac{1}{|\tilde{\Omega}|}\int_{\tilde{\Omega}}(u-v)^p\Big)^{\frac{1}{p}} \leq C \inf_{\tilde{\Omega}}(u-v).\label{eq:harn}
\end{equation} 
\end{proposition}
\begin{proof}
Provided we are able to show $u-v$ is a supersolution of a homogeneous linear elliptic PDE this is a consequence of the weak Harnack inequality \cite[Theorem 9.22]{gt} and a covering argument. To apply the Harnack inequality to $u-v$ we recall $C^{1,1}(\Omega)\subset W^{2,\infty}_{\text{loc}}(\Omega)$ \cite[Theorem 4.5]{EvansGariepyBook}. We now show $w:=u-v$ satisfies
\begin{equation}
 Lw := a^{ij}D_{ij}w+b^kD_kw+cw \leq 0,\label{eq:lin}
\end{equation}
where
\begin{align*}
  a^{ij} &=  [D^2u-A(\cdot,u,Du)]^{ij},\\
b^i &= -a^{ij}(A_{ij})_{p_k}-\tilde{B}_{p_k},\\
c &= -a^{ij}(A_{ij})_u-\tilde{B}_u,
\end{align*}
and $\tilde{B} = \log B$. Now using (\ref{eq:mate}) we have, almost everywhere, 
\begin{align}
 \label{eq:tosub}0=\log\det[D^2v&-A(\cdot,v,Dv)] - \log\det[D^2u-A(\cdot,u,Du)]\\
&+\tilde{B}(\cdot,u,Du) - \tilde{B}(\cdot,v,Dv).\nonumber
\end{align}
A Taylor series for 
\[ h(t) := \log\det[t(D^2v-A(\cdot,v,Dv))+(1-t)(D^2u-A(\cdot,u,Du))],\]
yields 
\[ h(1) - h(0) = h'(0) + \frac{1}{2}h''(\tau),\]
for some $\tau$ in $[0,1]$. Concavity of $\log\det$ implies $h''(\tau) \leq 0$ and thus on computing $h'(0)$ we obtain
\begin{align}
 \label{eq:tosub1} \log\det[D^2v&-A(\cdot,v,Dv)] - \log\det[D^2u-A(\cdot,u,Du)]\\
&\leq a^{ij}D_{ij}(v-u) + a^{ij}(A_{ij}(\cdot,u,Du) - A_{ij}(\cdot,v,Dv)),\nonumber
\end{align}
where $a^{ij} = [D^2u-A(\cdot,u,Du)]^{ij}$. Thus \eqref{eq:tosub1} into \eqref{eq:tosub} implies
\begin{align}
 \label{eq:tosub2}0\leq a^{ij}&D_{ij}(v-u) + a^{ij}(A_{ij}(\cdot,u,Du) - A_{ij}(\cdot,v,Dv))\\
&+\tilde{B}(\cdot,u,Du) - \tilde{B}(\cdot,v,Dv).\nonumber
\end{align}
The mean value theorem yields
\[ A_{ij}(\cdot,u,Du) - A_{ij}(\cdot,v,Dv) = A_u(\cdot,w,p)(u-v) +A_{p_k}(\cdot,w,p)D_k(u-v), \]
 for some $w = t_1 v+(1-t_1)u$ and $p = t_2Dv + (1-t_2)Du$ and similarly for 
$\tilde{B}(\cdot,u,Du) -\tilde{B}(\cdot,v,Dv)$. Thus \eqref{eq:tosub2} becomes
\[0 \leq a^{ij}D_{ij}(v-u)-(a^{ij}(A_{ij})_{p_k}+\frac{B_{p_k}}{B})D_k(v-u)-(a^{ij}(A_{ij})_u+\frac{B_u}{B})(v-u) ,\]
which is \eqref{eq:lin} (multiply by $-1$ since $w = u-v$).
\end{proof}

\section{Solutions intersecting on the interior are the same}\label{sec:tang}
Here we provide the proof of Theorem \ref{thm:main}. The Harnack inequality implies that one solution cannot touch another from above. Now we show that given two distinct solutions, since their derivatives are equal where they intersect, their maximum is a $C^{1,1}(\Omega)$ solution touching from above --- a contradiction. 
\begin{proof}[Proof: (Theorem \ref{thm:main})]
  At the outset we fix $\tilde{\Omega} \subset\subset \Omega$ containing $x_0$. Since $u,v$ are $g$-convex the same is true for $w:=\max\{u,v\}$. Furthermore $Du \equiv Dv$ on $\{u=v\}$, implies $w$  is $C^{1,1}(\Omega)$. Thus we obtain that $w$ solves (\ref{eq:mate}) almost everywhere. Hence we may apply our weak Harnack inequality \eqref{eq:harn} to $w-v$ to obtain $w \equiv v$ in $\tilde{\Omega}$. The same argument yields $w \equiv u$ in $\tilde{\Omega}$ and hence $u \equiv v$ in $\Omega$ via continuity.  \end{proof}

\section{Solutions intersecting on the boundary are the same}
\label{sec:boundary}
We conclude by proving that if solutions intersect on the boundary then they are the same throughout the domain. We require a convexity assumption on the target domain $\Omega^*$. We say a $C^2$ connected domain $\Omega^*$ is $Y^*$ convex with respect to  $x \in \Omega$ and $ u \in \R$ provided there exists a defining function $\phi^* \in C^2(\ol{\Omega^*})$ satisfying
\begin{align*}
  &\phi^* < 0 \text{ in } \Omega^* &&\phi = 0 \text{ on }\d\Omega^*\\
  & D_{p}^2\phi^*(Y(x,u,p)) \geq 0 && |D\phi|\neq 0 \text{ on }\d\Omega^*.
\end{align*}

For a comparison between this and other definitions of domain convexity see \cite[Section 2.2]{LT}. In the same paper Liu and Trudinger prove that for $C^2(\ol{\Omega})$ solutions and 
\[ G(x,u,p) := \phi^*(Y(x,u,p)) \] the boundary condition
\[ G(\cdot,u,Du) = 0 \text{ on }\d\Omega,\]
is oblique, i.e satisfies $G_p \cdot \gamma > 0$ where $\gamma$ is the outer unit normal.

\begin{theorem}
  Assume the conditions of Theorem \ref{thm:main}. In addition assume $u,v \in C^{2}(\ol{\Omega})$ are generalized solutions of (\ref{eq:pje}) subject to (\ref{eq:2bvp}) and $\Omega^*$ is $Y^*$-convex with respect to each $x \in \Omega$ and an interval containing $u(\Omega)\cup v(\Omega)$. If there is $x_0 \in \d \Omega$ with $u(x_0) = v(x_0)$ then $u \equiv v$ in $ \ol{\Omega}$. 
\end{theorem}
\begin{proof}
 Using Theorem \ref{thm:main} it suffices to prove there is $x \in \Omega$ with $u(x) = v(x)$.  For a contradiction suppose at some $x_0$ in $\d \Omega$ we have $u(x_0) = v(x_0)$, yet in $\Omega$ there holds $u > v$. Hopf's lemma (\cite[Lemma 3.4]{gt}) yields
\begin{equation}
  \label{eq:hopf}
  D_\gamma(u-v)(x_0) < 0.
\end{equation}
Here we used that the linear elliptic inequality \eqref{eq:lin} is uniformly elliptic under the assumption $u \in C^{2}(\ol{\Omega})$ and that no sign condition is needed on the lowest order coefficient in \eqref{eq:lin} since $u(x_0)-v(x_0)=0.$ 

Consider the function $h(t) :=G(x_0,u(x_0),tDv(x_0)+(1-t)Du(x_0))$. A Taylor series yields
\[ h(1) = h(0) + h'(0)+h''(\tau)/2,\]
for some $\tau \in [0,1]$. Since $u(x_0) = v(x_0)$ we have $h(1), h(0) =0$. Furthermore convexity implies $h''(\tau) \geq 0$ and hence
\[ 0 \geq h'(0) = G_p \cdot D(v-u),\]
or equivalently $0 \leq G_p\cdot D(u-v).$
Combined with obliqueness we have
\[  D_\gamma(u-v)(x_0) \geq 0,\]
which contradicts \eqref{eq:hopf} and thus establishes the result.
\end{proof}

\subsection*{Acknowledgement}
\label{sec:acknowledgements}

 I would like to thank Neil Trudinger for many valuable discussions.
\bibliographystyle{srtnumbered}
\bibliography{BibCaleRankinBAMSNoCrossing}

\end{document}